\documentclass[11pt]{amsart}
\usepackage[english]{babel}
\usepackage{amsfonts,latexsym,amsthm,amssymb,graphicx}
\usepackage{mathabx}
\usepackage[all]{xy}
\usepackage[usenames]{color}
\usepackage{tikz}
\usetikzlibrary{shapes}
\usepackage{xypic}
\usetikzlibrary{decorations.pathreplacing}
\usepackage{amsmath}
\usepackage{graphicx}
\usepackage[enableskew]{youngtab}
\usepackage[utf8]{inputenc}
\usepackage[english]{babel}
\usepackage{tikz}
\usetikzlibrary{arrows}
\usepackage{float}
\usepackage{amscd}
\usepackage[colorlinks=true, linkcolor=blue, citecolor=blue, urlcolor=blue]{hyperref}

\DeclareMathOperator{\Gr}{Gr}

\newcommand{\C}{{\mathbb C}}
\newcommand{\Z}{{\mathbb Z}}

\newcommand{\cO}{{\mathcal O}}

\DeclareMathOperator{\GL}{GL}
\DeclareMathOperator{\Sp}{Sp}
\newcommand{\IG}{\mathrm{IG}}

\newcommand{\e}{\mathbf{e}}

\CompileMatrices

\newtheorem{thm}{Theorem}[section]
\newtheorem{lemma}[thm]{Lemma}

\newtheorem{prop}[thm]{Proposition}

{\theoremstyle{defn} \newtheorem{defn}[thm]{Definition}}

{\theoremstyle{remark} \newtheorem{remark}[thm]{Remark}
\newtheorem{example}[thm]{Example}}

\begin{document}

\title{Conjecture $\mathcal{O}$ holds for the odd symplectic Grassmannian} \author{Changzheng Li}
\address{School of Mathematics, Sun Yat-sen University, Guangzhou 510275, P.R. China} \email{lichangzh@mail.sysu.edu.cn}
\author{Leonardo C. Mihalcea}
\address{ Department of Mathematics, 460 McBryde Hall, Virginia Tech University, Blacksburg VA 24060 USA} \email{lmihalce@math.vt.edu}
\author{Ryan Shifler}
\address{ Department of Mathematics, 460 McBryde Hall, Virginia Tech University, Blacksburg VA 24060 USA}\email{twigg@vt.edu}
\date{June 2, 2017}
\thanks{C. L. was   supported in part by the Recruitment Program of Global Youth Experts in China and   the NSFC Grant 11521101.}\thanks{L.M. was supported in part by NSA Young Investigator Award 98320-16-1-0013 and a Simons Collaboration grant.}

\begin{abstract} Let $\IG(k, 2n+1)$ be the odd-symplectic Grassmannian. Property $\mathcal{O}$, introduced by Galkin, Golyshev and Iritani for arbitrary complex, Fano manifolds $X$, is a statement about the eigenvalues of the linear operator obtained by the quantum multiplication by the anticanonical class of $X$. We prove that property $\cO$ holds in the case when $X= \IG(k, 2n+1)$ is an odd-symplectic Grassmannian. The proof uses the combinatorics of the recently found quantum Chevalley formula for $\IG(k, 2n+1)$, together with the Perron-Frobenius theory of nonnegative matrices. \end{abstract}
\maketitle

\section{Introduction}\label{s:intro} Fix $1 \le k \le n+1$ and let $\IG:= \IG(k, 2n+1)$ be the odd-symplectic Grassmannian. This is a smooth Fano algebraic variety parametrizing $k$ dimensional linear subspaces $V \subset \C^{2n+1}$ which are isotropic with respect to a skew-symmetric, bilinear form $\omega$ with kernel of dimension $1$; see \cite{mihai:odd,pech:quantum, mihalcea.shifler:qhodd}. The purpose of this paper is to prove Galkin, Golyshev and Iritani's Conjecture $\cO$ \cite[Conj. 3.1.2]{GGI} for the variety $\IG$, i.e.~to verify that {\em Property $\cO$} holds for $\IG$. We recall the precise statement, following \cite[\S 3]{GGI}.

Let $K:=K_{\IG}$ be the canonical bundle of $\IG$ and let $c_1(\IG):=c_1(-K) \in H^2(\IG) \simeq \Z$ be the anticanonical class. The quantum cohomology ring $(\mathrm{QH}^*(\IG), \star)$ is a graded algebra over $\Z[q]$, where $q$ is the quantum parameter and it has degree $2n+2-k$. Consider the specialization $H^\bullet(\IG):= \mathrm{QH}^*(\IG)_{|q=1}$ at $q=1$. The quantum multiplication by the first Chern class $c_1(\IG)$  induces an endomorphism $\hat c_1$ of the finite-dimensional vector space $H^\bullet(\IG)$:
 \[ y\in H^\bullet(\IG) \mapsto \hat c_1(y):= (c_1(\IG)\star y)|_{q=1} \/. \] Denote by $\delta_0:=\max\{|\delta|:\delta \mbox{ is an eigenvalue of } \hat c_1\}$. Then Property $\mathcal{O}$ states the following.
  \begin{enumerate}
    \item The real number $\delta_0$ is an eigenvalue of $\hat c_1$ of multiplicity one.
    \item If $\delta$ is any eigenvalue of $\hat c_1$ with $|\delta|=\delta_0$, then $\delta=\delta_0 \zeta$ for some $r$-th root of unity $\zeta\in \mathbb{C}$, where $r= 2n+2-k$ is the Fano index of $\IG$.
  \end{enumerate}

The property $\cO$ was conjectured to hold for any Fano, complex, manifold $X$ by Galkin, Golyshev and Iritani \cite{GGI}. (In that case one considers the even part of the quantum cohomology ring, and one does not necessarily restrict to $rank~H^2(X) = 1$.)~It is the main hypothesis needed for the statement of Gamma Conjectures I and II, which in turn are related to mirror symmetry on $X$ and refine Dubrovin conjectures; we refer to \cite{GGI} for details. Property $\cO$ was proved for several Grassmannians of classical types \cite{Riet,GaGo,Cheo} and a complete proof was recently given for any homogeneous space $G/P$ \cite{ChLi}.  

The odd-symplectic Grassmannian $\IG$ admits an action of Proctor's (complex) odd-symplectic group $\Sp_{2n+1}$ \cite{proctor:oddsymgrps}. If $k< n+1$ then $\Sp_{2n+1}$ acts with two orbits and if $k= n+1$ the action is transitive and $\IG$ is isomorphic to the Lagrangian Grassmannian $\IG(n, 2n)$. The odd-symplectic Grassmannian is sandwiched between two homogeneous spaces \[ \IG(k-1, 2n) \subset \IG(k, 2n+1) \subset \IG(k, 2n+2) \] where $\IG(k, 2n+2)$ parametrizes the $k$-dimensional subspaces in $\C^{2n+2}$ which are isotropic with respect to a symplectic form on $\C^{2n+2}$ (and similarly for $\IG(k-1, 2n)$). Then $\IG(k-1, 2n)$ can be identified with the closed orbit under $\Sp_{2n+1}$-action, while $\IG(k, 2n+1)$ is a smooth Schubert variety in $\IG(k,2n+2)$; see \cite{mihai:odd, pech:quantum} and \S \ref{s:prelims} below. An easy exercise is to check this for $k=1$: then $\IG(1, 2n+1) = \mathbb{P}^{2n}$ and the closed orbit consists of a single point.

Because quantum cohomology is not functorial, one needs to check Property $\cO$ on a case by case basis. In particular, its knowledge for the isotropic Grassmannians $\IG(k-1,2n)$ and $\IG(k, 2n+2)$ does not imply it for the odd-symplectic Grassmannians $\IG$. Our proof is based on the Perron-Frobenius theory of non-negative matrices, applied to the operator $\hat c_1$. The usefulness of this theory for proving Property $\mathcal{O}$ was already noticed in \cite[Rmk 3.1.7]{GGI}, and it was the main technique used by Cheong and Li \cite{ChLi}. The arguments from \cite{ChLi} use that the Gromov-Witten (GW) invariants for $G/P$ are enumerative, in particular the (Schubert) structure constants of $\mathrm{QH}^*(G/P)$ are non-negative integers, and in addition, that the GW invariants satisfy certain symmetries. However, the positivity does not hold for the odd-symplectic Grassmannian (see e.g.~(\ref{E:qneg}) below), and it is still unknown whether any analogous symmetries exist. We circumvent this problem by making heavy use of the combinatorics of the recently found quantum Chevalley formula in $\mathrm{QH}^*(\IG)$ \cite{mihalcea.shifler:qhodd}, which governs the quantum multiplication by $c_1(\IG)$.

\subsection*{Acknowledgements} The first named author thanks Daewoong Cheong for discussions and collaborations on related projects.

\section{Preliminaries}\label{s:prelims} In this section we introduce briefly the odd-symplectic Grassmanian and some basic properties of its cohomology ring. We refer to \cite{pech:thesis,pech:quantum} for details; we follow closely the exposition from \cite{mihalcea.shifler:qhodd}.

 Let $E:= \C^{2n+1}$ be an odd dimensional complex vector space with basis $\{ \e_1, \e_2, \ldots , \e_{2n+1} \}$. An odd-symplectic form is a skew symmetric, bilinear form $\omega$ on $E$ with kernel of dimension $1$. Without loss of generality, one can assume that $\ker \omega = \langle \e_1 \rangle$ and that $\omega (\e_i, \e_{2n+3-j} )= \delta_{i,j}$ for $1 \le i \le n+1$ and $2 \le j \le n+1$.
The {\em odd-symplectic Grassmannian} $\IG:=\IG(k, 2n+1)$ parametrizes subspaces of dimension $k$ in $E$ which are isotropic with respect to the form $\omega$. It is naturally a subspace of the ordinary Grassmannian $\Gr(k, 2n+1)$, and it is in fact the zero locus of a general section on $\bigwedge^2 \mathcal{S}^*$ induced by the symplectic form $\omega$; here $\mathcal{S}$ denotes the rank $k$ tautological subbundle on $\Gr(k, 2n+1)$. As such it is a projective manifold of dimension \[ \dim \IG = \dim \Gr(k, 2n+1) - \frac{k(k-1)}{2} = k (2n+1-k) - \frac{k(k-1)}{2} \/. \] The form $\omega$ can be completed to a non-degenerate form $\widetilde{\omega}$ on a space $\C^{2n+2}$, and this gives an embedding $\iota: \IG \to \IG(k, 2n+2)$ into the symplectic Grassmannian which parametrizes linear subspaces isotropic with respect to $\widetilde{\omega}$. The restriction of $\omega$ to the subspace $\C^{2n} = \{ \e_2, \ldots , \e_{2n+1} \}$ is non-degenerate, and this gives an inclusion $\IG(k, 2n) \to \IG(k, 2n+1)$ of the symplectic Grassmannian $\IG(k, 2n)$ into the odd-symplectic one. Therefore one can regard the odd-symplectic Grassmannian as an ``intermediate" space between two symplectic Grassmannians. More is true: the symplectic Grassmannians $\IG(k, 2n)$ and $\IG(k, 2n+2)$ are homogeneous spaces for the (complex) symplectic groups $\Sp_{2n}$ and $\Sp_{2n+2}$ respectively. The odd-symplectic Grassmannian has an action of the {\em odd-symplectic group} $\Sp_{2n+1}$, defined by Proctor \cite{proctor:oddsymgrps,proctor:oddsymgrpscomb}. This group contains $\Sp_{2n}$ as a subgroup, and it is contained in $\Sp_{2n+2}$ (but {\em not} as a subgroup). By definition, the odd-symplectic group is the subgroup of $\GL_{2n+1}(\C)$ consisting
of those $g \in \GL_{2n+1}(\C)$ such that $\omega(g.u, g.v) = \omega(u,v)$ for any $u,v \in \C^{2n+1}$. If $k < n+1$ the group $\Sp_{2n+1}$ acts on $\IG$ with two orbits given by: \[ X^\circ:=  \{V \in \IG(k, 2n+1): \e_1 \notin V \}; \quad X_c := \{ V \in \IG(k, 2n+1): \e_1 \in V \} \/. \] Notice that the closed orbit $X_c$ can be naturally identified with $\IG(k, 2n)$. We also remark that if $k=1$ then $\IG(1, 2n+1)= \mathbb{P}^{2n}$ (the projective space), while if $k=n+1$ then $\IG(k, 2n+1) = \IG(n, 2n)$ is the Lagrangian Grassmannian; in the first situation $X_c =  \{ \langle \e_1 \rangle \}$ (a point), and in the second $X_c = \IG(n,2n)$, thus $X^\circ = \emptyset$.
{\em To avoid special cases, from now on we will consider $k < n+1$.}

Let $P \subset \Sp_{2n+2}$ be the maximal parabolic subgroup which preserves $\e_1$ (i.e.the kernel of $\omega$) and let $B_{2n+2} \subset \Sp_{2n+2}$ be the Borel subgroup which preserves the standard flag in $\C^{2n+2}$. Mihai showed in \cite[Prop. 3.3]{mihai:odd} that there is a surjection $P \to \Sp_{2n+1}$ obtained by restricting $g \mapsto g_{|E}$. Then the Borel subgroup of $B_{2n+2}$ restricts to the (Borel) subgroup $B \subset \Sp_{2n+1}$. We recall the description of $B$-orbits on $\IG$.

A Schubert variety in $\IG(k, 2n+2)$ is the closure of an orbit of the Borel subgroup $B_{2n+2}$. We follow conventions from \cite{BKT2} and index these Schubert varieties by $(n-k+1)$-strict partitions of the form $\lambda=(\lambda_1 \ge \lambda_2 \ge \ldots \ge \lambda_k)$, where $\lambda_1 \le 2n+2 -k $ and $\lambda_k \ge 0$; the $(n+1-k)$-strict condition means that $\lambda_i > \lambda_{i+1}$ whenever $\lambda_i > n-k+1$. We denote the set of these partitions by $\Lambda^{ev}$. For $1 \le i \le n+1$ define $F_i: = \langle \e_1, \ldots , \e_i \rangle $ and $F_{n+1+i} = F_{n+1-i}^\perp$, where the perp is taken with respect to the completed (non-degenerate) symplectic form $\widetilde{\omega}$. The Schubert variety $Y(\lambda) \subset \IG(k, 2n+2)$ relative to the isotropic flag $F_\bullet = (F_i)$ is defined by
\[ Y(\lambda):= \{ V \in \IG(k,2n+2) : \dim (V \cap F_{w(j)}) \geq j \mbox{ } \forall 1 \leq j \leq \ell(\lambda) \}\] where $\ell(\lambda)$ is the number of non-zero parts of $\lambda$ and \[w(j)=2n+3-k-\lambda_j+\# \{i<j: \lambda_i+\lambda_j :2(n-k)+j-i \}. \]
This is a subvariety of $\IG(k, 2n+2)$ of codimension $|\lambda| := \lambda_1 + \ldots + \lambda_k$. Let $1^k$ denote the partition $(1, \ldots , 1)$. The following key fact is due to Mihai \cite{mihai:thesis,mihai:odd}.

\begin{thm}\label{thm:oddSchubert} (a) The odd-symplectic Grassmannian $\IG(k, 2n+1)$ equals the Schubert variety $Y(1^k)$ in $\IG(k, 2n+2)$.

(b) Those Schubert varieties $Y(\lambda)$ of $\IG(k, 2n+2)$ contained in $Y(1^k)$ coincide with the closures of the orbits of the odd-symplectic Borel group $B$ acting on $\IG(k, 2n+1)$.

\end{thm}
The theorem allows us to define the Schubert varieties in $\IG(k, 2n+1)$ as the Schubert varieties in $\IG(k, 2n+2)$ contained in $Y(1^k)$. One can check that $Y(\lambda) \subset Y(1^k)$ if and only if $\lambda$ satisfies the condition that if $\lambda_k = 0$ then $\lambda_1 = 2n+2 - k$; in other words, if the first column is not full, then the first row must be full.\begin{footnote} {One word of caution: the Bruhat order does not translate into partition inclusion. For example, $(2n+2-k, 0, \ldots , 0) \le (1, 1, \ldots 1) $ in the Bruhat order for $k < n+1$.}\end{footnote} We will use a variant of the indexing set $\Lambda^{ev}$, due to P{e}ch, which conveniently records the codimension relative to $\IG$: \[ \Lambda := \{ \lambda =  (2n+1-k \geq\lambda_1 \geq \cdots \geq \lambda_k \geq -1): \lambda \textrm{ is } n-k \textrm{-strict}, \textrm{ if } \lambda_k=-1 \textrm{ then } \lambda_1=2n+1-k \} \/. \] Pictorially, the partitions in $\Lambda$ are obtained by removing the full first column $1^k$ from the partitions in $\Lambda^{2n+2}_k$, regardless of whether a part equal to $0$ is present. For $\lambda \in \Lambda$, we define the Schubert variety $X(\lambda):= Y(\lambda+ 1^k)$. Then the codimension of $X(\lambda)$ in $\IG(k, 2n+1)$ equals $|\lambda|$.

\begin{example}

Let $k=5$, $n=7$. Consider the partition $(11 \geq 6 \geq 3 \geq 3 \geq 0) \in \Lambda^{ev}$. The corresponding partition in $\Lambda$ is \[(\lambda_1-1 \geq \lambda_2-1 \geq \lambda_3-1\geq \lambda_4-1\geq \lambda_5-1)=(10 \geq 5 \geq 2 \geq 2 \geq -1).\]
Pictorially,
\begin{equation*}
\Yvcentermath1
\yng(11,6,3,3,0)   -  \yng(1,1,1,1,1)=
\young(:\strut \strut \strut \strut \strut \strut \strut \strut \strut \strut ,:\strut \strut \strut \strut \strut,:\strut \strut,:\strut \strut,\strut)
\end{equation*}

\end{example}

\subsection{The (quantum) cohomology ring} The theorem \ref{thm:oddSchubert} implies that the cohomology ring $H^*(\IG)$ of $\IG$ has a $\Z$-basis given by the fundamental classes of Schubert varieties $[X(\lambda)] \in H^{2 |\lambda|}(\IG)$ where $\lambda$ varies in $\Lambda$. The quantum cohomology ring $\mathrm{QH}^*(\IG)$ is a graded $\Z[q]$-algebra with a $\Z[q]$-basis given by Schubert classes $[X(\lambda)]$ for $\lambda \in \Lambda$. The grading is given by $\deg q = 2n+2-k$ (i.e. the degree of the anticanonical divisor). The multiplication is given by \[ [X(\lambda)] \star [X(\mu)] = \sum_{\nu \in \Lambda, d \ge 0} c_{\lambda, \mu}^{\nu,d} q^d [X(\nu)] \/, \] where $c_{\lambda, \mu}^{\nu,d}$ are the $3$-point, genus $0$, Gromov-Witten invariants corresponding to rational curves of degree $d$ intersecting the classes $[X(\lambda)]$, $[X(\mu)]$ and the Poincar{\'e} dual of $[X(\nu)]$. Unlike the homogeneous case, these numbers might be negative in general. P{e}ch found a quantum Pieri rule $[X(\lambda)] \star [X(i)]$ in the case of the odd-symplectic grassmannian of lines $\IG(2, 2n+1)$. She proved that in $\mathrm{QH}^*(\IG(2,5))$, \begin{equation}\label{E:qneg} [X(3,-1)] \star [X(2,1)] = -1 [X(3,2)] + \ldots ; \quad [X(3,-1)] \star [X(3,-1)] = - q [X(0)] + \ldots \/. \end{equation} For arbitrary $k$, the second and third named authors found a Chevalley formula calculating $[X(1)] \star [X(\lambda)]$ in the {\em equivariant} quantum cohomology ring, and proved that this formula gives a recursive algorithm to calculate all the other structure constants; see \cite{mihalcea.shifler:qhodd}.

For the purpose of this paper, we will need the multiplication in the quantum cohomology ring by the Chern class of the anticanonical line bundle $-K:= -K_{\IG}$, i.e. by the first Chern class $c_1(\IG(k, 2n+1))$ of the odd-symplectic Grassmannian. A standard calculation yields \[  c_1(\IG) = (2n+2-k) [X(1)] \/,  \] therefore the quantum multiplication by $c_1(\IG)$ is governed by the Chevalley formula $[X(\lambda)] \star [X(1)]$.
Notice also that $X(1)$ is an ample divisor in $\IG$ (it is simply the restriction of the Schubert divisor from $\Gr(k, 2n+1)$), therefore $\IG(k, 2n+1)$ is a Fano variety. We refer to either \cite{pech:thesis} or \cite{mihalcea.shifler:qhodd} for details. To describe the quantum multiplication by $[X(1)]$ we need to recall the ordinary Chevalley formula in $H^*(\IG(k, 2n+2))$ proved in \cite{BKT2}.

\begin{defn} \label{Def:->}Let $\lambda \in \Lambda^{ev}$. Following \cite[Definitions 1.2, 1.3 ]{BKT2} we say that the box in row $r$ and column $c$ of $\lambda$ is $(n+1-k)${\it -related} to the box in row $r'$ and column $c'$ if \[|c-n+k-2|+r=|c'-n+k-2|+r'. \]
Given $\lambda, \mu \in \Lambda^{ev}$ with $\lambda \subset \mu$, the skew diagram $\mu / \lambda$ is called a horizontal strip (resp. vertical) strip if it does not contain two boxes in the same column (resp. row).

We say that $ \xymatrix{\lambda \ar[r]^{ev} & \mu}$ for any $(n+1-k)$-strict partitions $\lambda, \mu$ if $\mu$ can be obtained by removing a vertical strip from the first $n+1-k$ columns of $\lambda$ and adding a horizontal strip to the result, so that
\begin{enumerate}
\item if one of the first $n+1-k$ columns of $\mu$ has the same number of boxes as the column of $\lambda$, then the bottom box of this column is $(n+1-k)$-related to at most one box of $\mu \setminus \lambda$; and
\item if a column of $\mu$ has fewer boxes than the same column of $\lambda$, the removed boxes and the bottom box of $\mu$ in this column must each be $(n+1-k)$-related to exactly one box of $\mu \setminus \lambda$, and these boxes of $\mu \setminus \lambda$ must all lie in the same row.
\end{enumerate}
If $ \xymatrix{\lambda \ar[r]^{ev} & \mu}$, we let $\mathbb{A}$ be the set of boxes of $\mu \setminus \lambda$ in columns $n+2-k$ through $2n+1-k$ which are not mentioned in (1) or (2). Then define $N(\lambda, \mu)$ to be the number of connected components of $\mathbb{A}$ which do not have a box in column $n+2-k$. Here two boxes are connected if they share at least a vertex.
\end{defn}

\begin{example}\label{Ex:dominate} If $\mu$ is obtained from $\lambda$ by adding exactly one box, then $ \xymatrix{\lambda \ar[r]^{ev} & \mu}$. A more interesting example is when $\lambda = (2n-2k+2, 1, \ldots , 1)$ (with $k-1$ ones) and $\mu = (2n+2 -k)$. In this case each of the boxes in column one of $\lambda$ is $(n+1-k)$-related to exactly one box in the first row and last $k$ columns of $\lambda$. For instance, consider the case $n=8$ and $k=6$. The related boxes are shown in the figure below.
\begin{center} \begin{figure}[h!]\includegraphics{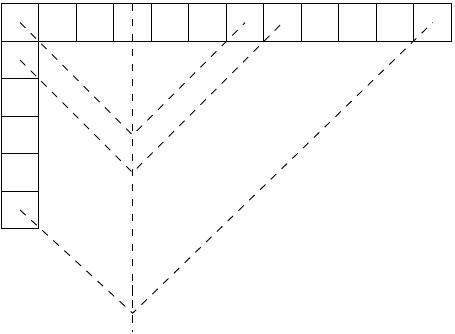}\end{figure}\end{center}
\end{example}
\begin{defn} Let $\lambda, \mu \in \Lambda$ be two partitions associated to the odd-symplectic Grassmannian $\IG(k, 2n+1)$. We say that $ \lambda \rightarrow \mu$ if $ \xymatrix{\lambda+1^k  \ar[r]^{ev} & \mu+1^k}$. If this is the case, we denote by $A(\lambda, \mu):= N(\lambda+1^k, \mu+1^k)$. \end{defn}

We need one more definition, for the partitions which will  appear as quantum terms.

\begin{defn}\label{def:lambda*} Let $\lambda = (\lambda_1, \ldots , \lambda_k)$ be a partition in $\Lambda$ such that $\lambda_1 = 2n+1-k$.

(a) If $\lambda_k \geq 0$ then let $\lambda^{*}=(\lambda_2 \geq \lambda_3 \geq \cdots \geq \lambda_k \geq 0)$. If $\lambda_k=-1$ then $\lambda^{*}$ does not exist.

(b) If $\lambda_2=2n-k$ then let $\lambda^{**}=(\lambda_1 \geq \lambda_3 \geq \cdots \geq \lambda_k \geq -1)$. If $\lambda_2<2n-k$ then $\lambda^{**}$ does not exist.
\end{defn}
In both situations notice that $|\lambda^*| = |\lambda^{**}| =  |\lambda| -( 2n+1 - k)$. As an example, if $\rho = (2n-k+1, 2n-k, \ldots , 2n-2k+2)$ is the partition indexing the Schubert point, then
$\rho^{*}= (2n-k, \ldots, 2n-2k+2, 0)$ and $\rho^{**} = (2n+1-k, 2n-1-k, \ldots, 2n-2k+2, -1)$. Clearly, there are also examples when one of the partitions $\lambda^*$ or $\lambda^{**}$ does not exist, but the other does. We are now ready to state the Chevalley formula proved in \cite{mihalcea.shifler:qhodd} (see also \cite{pech:quantum} for the case $k=2$).

%%%%%%%%%%%%%%%%%%%%%%%%%%%%%
\begin{thm}[quantum Chevalley formula for $\IG(k,2n+1)$]\label{thm:Chevalley} Let $\lambda \in \Lambda$ a partition. Then the following holds in $\mathrm{QH}^*(\IG)$:
\begin{eqnarray*}
[X(1)] \star [X(\lambda)]= \sum_{\lambda \rightarrow \mu, |\mu|= |\lambda|+1} 2^{A(\lambda, \mu)} [X(\mu)]+q[X(\lambda^*)]+q[X(\lambda^{**})]
\end{eqnarray*}
 If $\lambda^*$ or $\lambda^{**}$ do not exist then the corresponding quantum term is omitted.
\end{thm}
Consider the operator \begin{equation}\label{E:T} \displaystyle T= \sum_{i=0}^{\dim \IG} c_1(\IG)^{\star i}_{|q=1} = \sum_{i=0}^{\dim \IG} (2n+2-k)^i [X(1)]^{ \star i}_{|q=1}\/, \end{equation}
%$\displaystyle T=\sum_{i=0}^{\dim \IG(k,2n+1)} (2n+2-k)^i [X_{(1)}]^{*i}$
acting on $H^\bullet(\IG)= \mathrm{QH}^*(\IG)_{|q=1}$. From the Chevalley formula it follows that if $\lambda \in \Lambda$ is an arbitrary partition then $T[X(\lambda)]$ is an effective combination of Schubert classes. We say that $T[X(\lambda)] >0$ if in the expansion \[ T[X(\lambda)]= \sum_{\mu \in \Lambda}  a(\lambda,\mu) [X(\mu)] \] all coefficients are strictly positive, i.e. $a(\lambda,\mu)>0$ for all $\mu \in \Lambda$. Next is a key result in this paper.

\begin{thm}\label{thm:positive} Let $\rho= (2n-k+1, 2n-k, \ldots , 2n-2k+2)$ be the partition indexing the class of the point. Then the following positivity properties hold:

(a) For any $\lambda \in \Lambda$, the coefficient of $[X(\rho)]$ in $T[X(\lambda)]$ is strictly positive;

(b) The coefficient of $[X(0)]$ in $T[X(\rho)]$ is strictly positive;

(c) $T[X(0)] >0$.

\end{thm}
Before the proof of the theorem we recall the notion of the (oriented) {\em quantum Bruhat graph} of $\IG$; see \cite{brenti.fomin}. The vertices of this graph consist of partitions $\lambda \in \Lambda$. There is an oriented edge $\lambda \rightarrow \mu$ if the class $[X(\mu)]$ appears with positive coefficient (possibly involving $q$) in the quantum Chevalley multiplication $[X(\lambda)] \star [X(1)]$. An oriented, quantum, {\em Chevalley chain} between two partitions $\lambda$ and $\mu$ is a chain
\[ \lambda_0:=\lambda \rightarrow \lambda_1 \rightarrow \ldots \rightarrow \lambda_s :=\mu \]
in the quantum Bruhat graph of $\IG(k, 2n+1)$.

\begin{remark} Theorem \ref{thm:positive} implies that the quantum Bruhat graph of $\mathrm{IG}$ is strongly connected, i.e. any two of its vertices can be connected by an oriented chain. It is natural to conjecture that $T[X(\lambda)] >0$ for any $\lambda$. If this conjecture is true, it implies that any two points can be connected by a chain containing at most $\dim \IG$ edges.\end{remark}

\begin{proof}[Proof of Theorem \ref{thm:positive}] In each of the parts (a) and (b) it suffices to produce a Chevalley chain between two appropriate partitions which involves at most $\dim \IG $ edges. We consider first the coefficient $a(\lambda, \rho)$. Clearly $a(\lambda, \lambda)\ge 1$, therefore we are done if $\lambda = \rho$. If not then one can keep adding exactly one box to produce a Chevalley chain from $\lambda$ to $\rho$. If $\lambda_k \ge 0$ then it is clear that we arrive at $\rho$ after adding at most $\dim \IG$ boxes. If $\lambda_k = -1$ then necessarily $\lambda_1 = 2n+1-k$, and in the worst case scenario (when $\lambda = (2n+1-k, -1, \ldots , -1)$) we need to add \[ (k-1) + (2n-k) + (2n-k-1) + \ldots + (2n-2k+2) = \dim \IG - (2n-2k+2) \] boxes.
We now turn to part (b). A Chevalley chain from $\rho$ to $(0)$ is constructed as follows. Let $\eta$ be a partition in $\Lambda$ of the form $\eta=(\eta_1 \geq \ldots \geq  \eta_k \geq 0)$. If $\eta_1=2n+1-k$ then $\eta':=\eta^{*}$ exists and it will be the successor of $\eta$. If $\eta_1<2n+1-k$ then the successor of $\eta$ is $\eta':=(\eta_1+1 \geq \eta_2 \geq \ldots \geq \eta_k)$. Notice that in both situations $\eta'_k \ge 0$.
Now start from $\eta:=\rho$ and continue with the rules above. All partitions $\eta$ in this chain will satisfy $\eta_k \ge 0$, and such a chain requires at most \[ k + \sum_{i=1}^{k-1}i= k + \frac{k(k-1)}{2}\] edges to get to the partition $(0)$.
%need $\sum_{i=1}^{k-1}i=\frac{(k-1)k}{2}$ steps to complete (2) and we will need $k$ steps to complete (1). So in total we will need $\frac{(k-1)k}{2}+k$ steps.
Since \[ \dim\IG=k(2n+1-k)-\frac{(k-1)k}{2} \ge \frac{(k-1)k}{2}+k \/,\] this completes the proof of (b).

To prove (c) we distinguish two cases, when $\lambda_k \ge 0$ and when $\lambda_k= -1$. If $\lambda_k \ge 0$ then a Chevalley chain from $(0)$ to $\lambda$ is constructed by successively adding exactly one box, filling rows $1$, then $2$, etc. Clearly such a chain has exactly $|\lambda| \le \dim \IG$ edges.
Assume now that $\lambda_k = -1$. We first construct a chain from $(0)$ to $\alpha:=(2n+1-k, -1, \ldots , -1)$ (where there are $k-1$ ones) by \[ (0) \rightarrow (1) \rightarrow \ldots \rightarrow (2n-2k+1) \rightarrow \alpha \/. \] The last arrow
 exists by Example \ref{Ex:dominate}. This chain contains $2n-2k+2$ edges. From $\alpha$ to $\lambda$ one can again add exactly one box at every step, resulting in a Chevalley chain with $|\lambda| - |\alpha| = |\lambda| - (2n-2k+2)$ edges. Concatenate the two chains to get a chain from $(0)$ to $\lambda$ containing $|\lambda|$ edges.
\end{proof}

For future use we also record the following lemma.

\begin{lemma}\label{lemma:cycle} There exists a Chevalley cycle of length $2n+2-k$ of the form
\[ (0) \rightarrow (1) \rightarrow \ldots \rightarrow (2n-k+1) \rightarrow (0) \/.\] \end{lemma}

\begin{proof} This is clear from the Chevalley formula.\end{proof}
 \begin{example}
Consider the case $n=5$ and $k=4$. The following illustrates the chain constructed in part (b). (We also include the Chevalley coefficients and quantum parameters.)
\begin{eqnarray*}
&(7,6,5,4)& \longrightarrow q(6,5,4,0) \longrightarrow q(7,5,4,0) \longrightarrow q^2(5,4,0,0) \longrightarrow q^2(6,4,0) \longrightarrow 2q^2(7,4,0)\\
&\longrightarrow& 2q^2(4,0,0) \longrightarrow 2q^2(5,0,0) \longrightarrow 2^2q^2(6,0,0) \longrightarrow 2^3q^2(7,0,0) \longrightarrow 2^3q^3(0,0,0) \/.
\end{eqnarray*}
\end{example}
\section{Conjecture $\mathcal{O}$ for $\IG(k, 2n+1)$} In this section we prove the main result of this paper:
\begin{thm}\label{mainthm} The odd symplectic Grassmannian $\IG(k, 2n+1)$ satisfies Property $\mathcal{O}$, i.e. the quantum multiplication by $c_1(\IG(k, 2n+1))$ satisfies the conditions (1) and (2) stated in \S \ref{s:intro} above. \end{thm}
Recall that the Fano index $r$ of $\IG$ equals the degree of $q$, i.e. $r= 2n+2-k$.
In what follows we fix an arbitrary ordering of the Schubert classes $\{[X_\lambda]\}_\lambda$, and let $M$ denote the matrix of $\hat c_1$ with respect to such an ordered basis. The quantum Chevalley rule implies that $M$ is a {\em nonnegative} matrix, i.e. all its coefficients are nonnegative. The theory of nonnegative matrices (see e.g. in \cite{Minc}) will play a fundamental role. We refer to \cite[\S 3.1 and \S 3.2]{ChLi} for more details and the context of the facts needed for the proof.
\begin{lemma}\label{lem-irred}
 The matrix $M$ is irreducible in the sense that $PMP^t$ is never of the form $\Big(\begin{array}{cc}
A&B\\
0&D
\end{array}\Big)$ for any  permutation matrix $P$,   where $A,D$ are square submatrices.
\end{lemma}
\begin{proof}
If there exists a permutation matrix $P$ such that $PMP^t$ is a block-upper triangular matrix, then so is
$\sum_{m=0}^{\dim X} PM^mP^t$. The matrix of the operator $T$ is nonnegative, and since $M$ is reducible it follows that (the matrix of) $T$ is again reducible. By a remarkable property of reducible nonnegative matrices
(see e.g. \cite[Remark 3.1, part (1)]{ChLi}) $T$ must preserve a proper {\em coordinate} subspace $V$. Let $[X(\lambda)] \in V$ be a Schubert class inside this subspace. Another remarkable property is that if a class $[X(\mu)]$ appears with positive coefficient in the expansion of $T[X(\lambda)]$ then $[X(\mu)] \in V$ (again see \cite[remark 3.1, part (2)]{ChLi}). But then by part (a) of Theorem \ref{thm:positive} it follows that $[X(\rho)] \in V$, part (b) implies that $[X(0)] \in V$, and part (c) implies that any other $[X(\mu)] \in V$. In particular, $V = H^\bullet(\IG)$, which is a contradiction.\end{proof}

According to Perron-Frobenius theory, any irreducible nonnegative matrix has a real, positive eigenvalue $\delta_0$ of multiplicity one such that for any other eigenvalue $\delta$ there is an inequality $\delta_0 \ge |\delta|$; cf. \cite[Thm. 1.4]{BP}. In order to show that $M$ satisfies Property $\mathcal{O}$ we need to study the {\em index of imprimitivity} of $M$. We recall next the relevant definitions, following \cite{ChLi}.

\begin{defn} (a) Let $A$ be a nonnegative irreducible $n \times n$ matrix with maximal eigenvalue $\delta_0$, and suppose that $A$ has exactly $h$ eigenvalues of modulus $|\delta_0|$. The number $h$ is called the {\it index of imprimitivity of $A$.}

(b) Two matrices $A=(a_{i,j})$ and $B=(b_{i,j})$ are said to have the same {\it zero pattern} if $a_{i,j}=0$ if and only if $b_{i,j}=0$.~A directed graph $D(A)$ is said to be {\it associated} with a nonnegative matrix $A$, if the adjacency matrix of $D(A)$ has the same zero pattern as $A$.

(c) Let $D$ be a strongly connected directed graph. The greatest common divisor of the lengths of all cycles in $D$ is called the {\it index of imprimitivity of $D$} .
\end{defn}

To any pair of a nonnegative matrix $A$ and an {\em ordered} basis one can define a directed graph $D(A)$ such that $D(A)$ is associated to $A$. This is done by replacing nonzero entries of $A$ by $1$, and considering the resulting matrix as the adjacency matrix of a directed graph; the direction of the arrows are determined by the ordering of the basis; see e.g. \cite[\S 3.2]{ChLi}. In our situation the graph $D(M)$ is simply the oriented, quantum Bruhat graph defined in the previous section. Next is a key result relating the index of imprimitivity of $M$ to that of an associated directed graph; see Theorems 3.2 and 3.3 of Chapter 4 of \cite{Minc}.
\begin{prop}\label{prop:directed} Let A be a nonnegative matrix and $D(A)$ the associated directed graph defined above. Then the following hold:

(a) $A$ is irreducible if and only if the associated directed graph $D(A)$ is strongly connected;

(b) If $A$ is irreducible, then the index of imprimitivity $h(A)$ of $A$ is equal to the index of imprimitivity $h(D(A))$ of the associated directed graph $D(A)$.
\end{prop}
\noindent\begin{proof}[Proof of Theorem \ref{mainthm}] Let $h$ denote the imprimitivity of matrix $M$. The eigenvalues of $\hat c_1$ are that of the matrix $M$. Since $M$ is an irreducible, nonnegative matrix, the results \cite[Theorems 1.4 and 2.20 of Chapter 2]{BP} imply the following two facts.
 \begin{enumerate}
   \item[(i)]The real number $\delta_0$ is an eigenvalue of $\hat c_1$ of multiplicity one.
   \item[(ii)] Denote by $\delta_1,\cdots, \delta_h$ all    eigenvalues of $M$ of modulus  $\delta_0$ with multiplicities counted. Then $\{{\delta_i\over \delta_0}~|~1\leq i\leq h\}$ are precisely the $h$-th roots of unity. \end{enumerate}
Part (i) proves condition (1) of the Property $\mathcal{O}$. To prove the second condition of Property $\mathcal{O}$ it suffices to show that $h=r$ (the Fano index). A general property of Fano manifolds shows that $r$ always divides $h$; see \cite[Remark 3.1.3]{GGI}. For the converse, notice that by Lemma \ref{lemma:cycle} there exists a cycle of length $r=2n+2-k$ in $D(M)$. Then $h $ divides $r$ by Proposition \ref{prop:directed}, hence $h=r$ and we are done.\end{proof}
%%%%%%%
%\bibliography{Oconj}
%\bibliographystyle{amsplain}
\def\cprime{$'$}
\providecommand{\bysame}{\leavevmode\hbox to3em{\hrulefill}\thinspace}
\providecommand{\MR}{\relax\ifhmode\unskip\space\fi MR }
% \MRhref is called by the amsart/book/proc definition of \MR.
\providecommand{\MRhref}[2]{%
  \href{http://www.ams.org/mathscinet-getitem?mr=#1}{#2}
}
\providecommand{\href}[2]{#2}

\end{document}